\documentclass[a4paper,11pt]{amsproc}

\usepackage{amsmath,amsthm,amssymb,mathtools,mathrsfs}
\usepackage{stmaryrd}
\usepackage{ascmac}
\usepackage{comment}
\usepackage{bm}
\usepackage{booktabs}
\usepackage{graphicx}
\usepackage{subcaption}

\usepackage{hyperref}

\usepackage{colortbl}
\usepackage{graphicx}
\usepackage{here}
\usepackage{time}
\usepackage[abbrev]{amsrefs}

\usepackage{xcolor}
\usepackage[capitalize,nameinlink,noabbrev,nosort]{cleveref}
\hypersetup{
	colorlinks=true,       
	linkcolor=blue,          
	citecolor=blue,        
	filecolor=blue,      
	urlcolor=blue,           
}

\usepackage{fullpage}

\makeatletter
\@namedef{subjclassname@2020}{%
  \textup{2020} Mathematics Subject Classification}
\makeatother

\newtheorem{theoremcounter}{Theorem Counter}[section]

\theoremstyle{definition}
\newtheorem{dfn}[theoremcounter]{Definition}

\newtheorem{exm}[theoremcounter]{Example}

\theoremstyle{plain}
\newtheorem{lem}[theoremcounter]{Lemma}
\newtheorem{prop}[theoremcounter]{Proposition}

\newtheorem{thm}[theoremcounter]{Theorem}

\numberwithin{equation}{section}

\newcommand{\Z}{\mathbb{Z}}
\newcommand{\Q}{\mathbb{Q}}

\newcommand{\Zp}[1]{\mathbb Z_{(#1)}}

\newcommand{\vps}{\vphantom{$\displaystyle\sum$}}

\begin{document}

\title[]{On integrality and asymptotic behavior\\ of the $(k, l)$-G\"{o}bel sequences} 

\author{Hibiki Gima}
\address{Department of Mathematics, Kyushu University, Motooka 744, Nishi-ku, Fukuoka 819-0395, Japan}
\email{gima.hibiki.100@s.kyushu-u.ac.jp}

\author{Toshiki Matsusaka}
\address{Faculty of Mathematics, Kyushu University, Motooka 744, Nishi-ku, Fukuoka 819-0395, Japan}
\email{matsusaka@math.kyushu-u.ac.jp}

\author{Taichi Miyazaki}
\address{Department of Mathematics, Kyushu University, Motooka 744, Nishi-ku, Fukuoka 819-0395, Japan}
\email{miyazaki.taichi.478@s.kyushu-u.ac.jp}

\author{Shunta Yara}
\address{Department of Mathematics, Kyushu University, Motooka 744, Nishi-ku, Fukuoka 819-0395, Japan}
\email{yara.shunta.444@s.kyushu-u.ac.jp}

\subjclass[2020]{Primary 11B37; Secondary 11B50}
\thanks{The second author was supported by JSPS KAKENHI Grant Numbers JP20K14292 and JP21K18141.}



\maketitle

\begin{abstract}
    Recently, Matsuhira, Matsusaka, and Tsuchida revisited old studies on the integrality of $k$-G\"{o}bel sequences and showed that the first 19 terms are always integers for any integer $k\ge 2$. In this article, we further explore two topics: Ibstedt's $(k,l)$-G\"{o}bel sequences and Zagier's asymptotic formula for the $2$-G\"{o}bel sequence, and extend their results.
\end{abstract}


\section{Introduction}

Sloane's collection of integer sequences~\cite{Sloane1973} inspired G\"{o}bel to introduce a sequence defined by the recursion:
\[
	g_n = \frac{1 + g_0^2 + g_1^2 + \cdots + g_{n-1}^2}{n}
\]
starting with $g_0 =1$. Despite the initial terms $(g_n)_n = (1, 2, 3, 5, 10, 28, 154, 3520, \ldots)$ appearing to follow an integer sequence pattern, the sequence's integrality was not immediately clear, raising doubts about its suitability for inclusion in Sloane's collection. In 1975, Lenstra resolved this problem by showing that
\[
	g_n \in \Z \Longleftrightarrow n \le 42.
\]
Unfortunately, G\"{o}bel's sequence is not an integer sequence. Nevertheless, it has since been registered as an intriguing exception in his collection, OEIS~\cite[A003504]{OEIS}. (About its history, see also~\cite{MMT2024}). 

After that, several aspects of G\"{o}bel's sequence have been investigated. In 1990, Ibstedt~\cite{Ibstedt1990} focused on an alternative recursion:
\[
	(n+1) g_{n+1} = g_n (n+g_n)
\]
with the initial value $g_1 = 2$, and introduced a generalization.

\begin{dfn}\label{dfn:kl-Goebel}
	For integers $k, l \ge 2$, the \emph{$(k,l)$-G\"{o}bel sequence} $g_{k,l}(n)$ is defined by the recursion:
	\[
		(n+1) g_{k,l}(n+1) = g_{k,l}(n) (n + g_{k,l}(n)^{k-1})
	\]
	with the initial value $g_{k,l}(1) = l$.
\end{dfn}

We can pose the same question as Lenstra: When does its integrality break? To address the question, we introduce the notation
\begin{align}
	N_{k,l} \coloneqq \inf \{n \in \Z_{> 0} : g_{k,l}(n) \not\in \Z\}
\end{align}
for any $k, l \ge 2$. Lenstra's result is stated as $N_{2,2} = 43$. Then, Ibstedt provided a method to compute the values of $N_{k,l}$ and presented the list of $N_{k,l}$ for $2 \le k, l \le 11$ as follows. Here, we extend the list to include cases where $2 \leq k, l \leq 17$. We will provide Mathematica codes to compute $N_{k,l}$ in \cref{s-app}.

\begin{table}[H]
\centering
\begin{tabular}{|c|cccccccccccccccc|}
 \hline
 $l\backslash k$ & 2 & 3 & 4 & 5 & 6 & 7 & 8 & 9 & 10 & 11 & 12 & 13 & 14 & 15 & 16 & 17 \vps \\ \hline
 2 & 43 & 89 & 97 & 214 & 19 & 239 & 37 & 79 & 83 & 239 & 31 & 431 & 19 & 79 & 23 & 827 \vps\\
 3& 7 & 89 & 17 & 43 & 83 & 191 & 7 & 127 & 31 & 389 & 109 & 431 & 7 & 79 & 83 & 683 \vps\\
 4 & 17 & 89 & 23 & 139 & 13 & 359 & 23 & 158 & 41 & 169 & 103 & 643 & 31 & 79 & 167 & 118 \vps\\
 5& 34 & 89 & 97 & 107 & 19 & 419 & 37 & 79 & 83 & 137 & 31 & 431 & 19 & 41 & 23 & 59 \vps\\
 6 & 17 & 31 & 149 & 269 & 13 & 127 & 23 & 103 & 71 & 239 & 41 & 431 & 31 & 79 & 23 & 499 \vps\\
 7 & 17 & 151 & 13 & 107 & 37 & 127 & 37 & 103 & 83 & 239 & 101 & 167 & 19 & 79 & 13 & 59 \vps\\
 8 & 51 & 79 & 13 & 214 & 13 & 239 & 17 & 163 & 71 & 239 & 41 & 431 & 31 & 79 & 13 & 118 \vps\\
 9 & 17 & 89 & 83 & 139 & 37 & 191 & 23 & 103 & 23 & 239 & 41 & 431 & 47 & 79 & 29 & 177 \vps\\
 10 & 7 & 79 & 23 & 251 & 347 & 239 & 7 & 163 & 41 & 239 & 53 & 251 & 7 & 251 & 23 & 59 \vps\\
 11 & 34 & 601 & 13 & 107 & 19 & 461 & 37 & 79 & 31 & 389 & 101 & 479 & 19 & 79 & 13 & 59 \vps\\
 12 & 17 & 197 & 97 & 263 & 37 & 191 & 17 & 79 & 41 & 263 & 82 & 167 & 29 & 79 & 53 & 59 \vps\\
 13 & 17 & 151 & 23 & 263 & 37 & 127 & 37 & 158 & 31 & 137 & 61 & 431 & 19 & 41 & 83 & 271 \vps\\
 14 & 43 & 158 & 67 & 139 & 37 & 191 & 23 & 158 & 41 & 239 & 29 & 383 & 29 & 79 & 23 & 683 \vps\\
 15 & 67 & 197 & 173 & 139 & 37 & 239 & 37 & 127 & 31 & 1097 & 82 & 431 & 31 & 419 & 23 & 347 \vps\\
 16 & 59 & 151 & 157 & 107 & 59 & 359 & 37 & 103 & 46 & 137 & 29 & 431 & 29 & 79 & 23 & 607 \vps\\
 17 & 7 & 89 & 67 & 43 & 13 & 127 & 7 & 179 & 41 & 263 & 31 & 431 & 7 & 79 & 59 & 59 \vps\\ \hline
\end{tabular}
\caption{The list of $N_{k,l}$. As noted in OEIS~\cite[A097398]{OEIS}, the articles of Ibstedt~\cite{Ibstedt1990} and Guy~\cite[E15]{Guy2004} contain some mistakes in the values.}
\label{tab:Ibstedt}
\end{table}

In 1996, Zagier~\cite{Zagier1996} considered the asymptotic behavior of G\"{o}bel's sequence and described it as
\begin{align}\label{Zagier-asymp}
	g_{2,2}(n) \sim C^{2^n}n \left(1+\frac{2}{n}-\frac{1}{n^2}+\frac{4}{n^3}-\frac{21}{n^4}+\frac{138}{n^5}-\frac{1091}{n^6}+\cdots\right) \quad (n \to \infty)
\end{align}
without proof, where $C=1.0478314475764112295599\ldots$ is a constant. Finch~\cite[Section 6.10]{Finch2003} notes that this asymptotic formula shares the same coefficients as that for the sequence $s_n$ introduced by Somos. Here, the sequence $s_n$ is defined by the recursion:
\begin{align}\label{somos-seq}
	s_n = n s_{n-1}^2 
\end{align}
starting with $s_0 = 1$. Then, it satisfies that
\[
	s_n \sim \sigma^{2^n} n^{-1} \left(1+\frac{2}{n}-\frac{1}{n^2}+\frac{4}{n^3}-\frac{21}{n^4}+\frac{138}{n^5}-\frac{1091}{n^6}+\cdots\right)^{-1} \quad (n \to \infty),
\]
where
\begin{align}
	\sigma \coloneqq \prod_{n=1}^\infty n^{1/2^n} = 1.6616879496 \ldots
\end{align}
is called the \emph{Somos constant}.

After a period of silence, in 2023, inspired by \cite{KobayashiSeki2023}, Matsuhira, Tsuchida, and the second author~\cite{MMT2024} addressed the problem of determining the minimum value of $N_{k,2}$ and showed that
\begin{align}
	\min_{k \ge 2} N_{k,2} = 19.
\end{align}
Once again, G\"{o}bel's sequence returned as a subject of research.

In this article, we combine the above results and extend them as suggested by the previous work~\cite{MMT2024} in the last remarks. First, we consider the minimum value of all $N_{k,l}$ and show the following.

\begin{thm}\label{Main1}
	We have $\min_{k, l \ge 2} N_{k,l} = 7$, which implies that $g_{k,l}(n) \in \Z$ for any $k, l \ge 2$ and $1 \le n \le 6$. Moreover, we have $N_{k,l} = 7$ if and only if $k \equiv 2 \pmod{6}$ and $l \equiv 3 \pmod{7}$.
\end{thm}

Secondly, we give a complete proof of Zagier's asymptotic formula~\eqref{Zagier-asymp} and generalize it for any $(k,l)$-G\"{o}bel sequences. Before stating the theorem, we recall the definition of asymptotic expansions.

\begin{dfn}
	Assume that the sequence $(\lambda_r(n))_r$ satisfies $\lambda_{r+1}(n) = o(\lambda_r(n))$ as $n \to \infty$, that is,
	\[
		\lim_{n \to \infty} \frac{\lambda_{r+1}(n)}{\lambda_r(n)} = 0
	\]
	for any $r$. For a sequence $(c_n)_n$, we call $(a_r)_r$ its asymptotic coefficients and write
	\[
		c_n \sim \sum_{r=0}^\infty a_r \lambda_r(n)
	\]
	if 
	\[
		c_n - \sum_{r=0}^R a_r \lambda_r(n) = O(\lambda_{R+1}(n)) \quad (n \to \infty)
	\]
	holds for any $R \ge 0$.
\end{dfn}

By adapting the sequence $\lambda_r(n) = C_{k,l}^{k^n} n^{\frac{1}{k-1}} n^{-r}$, we obtain the following asymptotic expansion.

\begin{thm}\label{thm:generalized-zagier-asymptotic}
	For any integers $k, l \ge 2$, there exist a constant $C_{k,l}>1$ and a sequence $(a_{k,r})_r$ such that
	\begin{align*}
		g_{k,l}(n)\sim C_{k,l}^{k^n} n^{\frac{1}{k-1}} \left(1+\sum_{r=1}^\infty\frac{a_{k,r}}{n^r}\right) \quad (n \to \infty). 
	\end{align*}
\end{thm}

The constant $C_{k,l}$ and the sequence $(a_{k,r})_r$ will be explicitly defined in \eqref{dfn:C-kl} and \eqref{def-akr}, respectively.

In \cref{s2} and \cref{s3}, we give proofs of the above theorems, respectively. In \cref{s4}, we provide further observations on a variability of $g_{k,l}(n)$ modulo a higher power of $p$.

\section{How long can $(k,l)$-G\"{o}bel sequences remain integers?} \label{s2}

In this section, we provide a proof of \cref{Main1}, drawing on Ibstedt's method for computing $N_{k,l}$ and the argument presented by Matsuhira--Matsusaka--Tsuchida~\cite{MMT2024}. First, we prepare some notations

\subsection{Notations and key properties}

Let $\mathcal{P}$ be the set of all prime numbers. For each $p \in \mathcal{P}$, we let $\Zp{p}$ denote the localization of $\Z$ at the prime ideal $(p)$, that is, $\Zp{p} = \{a/b \in \Q : p \nmid b\}$. By the fact that
\begin{align}\label{Z-Zp}
	\bigcap_{p \in \mathcal{P}} \Zp{p} = \Z,
\end{align}
for $x \in \Q$, we have $x \in \Z$ if and only if $x \in \Zp{p}$ for all $p \in \mathcal{P}$. Let $\nu_p(x)$ be the $p$-adic valuation of $x \in \Q$. More precisely, for an integer $n$, $\nu_p(n)$ is the exponent of the largest power of $p$ that divides $n$, and it is extended to rational numbers by $\nu_p(a/b) = \nu_p(a) - \nu_p(b)$. The following lemma is the key for our proof of \cref{Main1}. For convenience, we will also include the cases when $k=1$ or $l=1$, in which $g_{k,l}(n)$ is a constant sequence.

\begin{lem}\label{lem:key-properties}
	Let $N$ be a positive integer. For each $p \in \mathcal{P}$, we put $r = \nu_p(N!)$. Let $k, k_1, k_2, l, l_1, l_2 \ge 1$ and $1 \le n \le N$ be integers. Then we have the following.
	\begin{enumerate}
		\item If $p > N$, then $g_{k,l}(N) \in \Zp{p}$.
		\item If $g_{k,l}(n) \not\in \Zp{p}$, then $g_{k,l}(n+1) \not\in \Zp{p}$.
		\item Assume that $p \le N$, that is, $r \ge 1$. If $k_1, k_2 \ge r$ and $k_1 \equiv k_2 \pmod{\varphi(p^r)}$, then $g_{k_1, l}(n) \in \Zp{p}$ if and only if $g_{k_2, l}(n) \in \Zp{p}$, where $\varphi(n)$ is the Euler totient function. Moreover, in this case, we have $g_{k_1, l}(n) - g_{k_2, l}(n) \in p^{r-\nu_p(n!)} \Zp{p}$.
		\item Assume that $p \le N$. If $l_1 \equiv l_2 \pmod{p^r}$, then $g_{k,l_1}(n) \in \Zp{p}$ if and only if $g_{k, l_2}(n) \in \Zp{p}$. Moreover, in this case, we have $g_{k,l_1}(n) - g_{k,l_2}(n) \in p^{r-\nu_p(n!)} \Zp{p}$.
	\end{enumerate}
\end{lem}

\begin{proof}
	(1) It is obvious from \cref{dfn:kl-Goebel}. 
	
	(2) If $g_{k,l}(n) \not\in \Zp{p}$, that is, $\nu_p(g_{k,l}(n)) < 0$, then we have
	\begin{align*}
		\nu_p(g_{k,l}(n+1)) &= \nu_p(g_{k,l}(n)) + \nu_p(n+g_{k,l}(n)^{k-1}) - \nu_p(n+1)\\
			&= k \nu_p(g_{k,l}(n)) - \nu_p(n+1) < 0.
	\end{align*}
	
	(3) It follows from Euler's theorem, (a generalization of Fermat's little theorem), and induction on $n$. For the initial case, we have $g_{k_1, l}(1) = l = g_{k_2, l}(1)$. Assume that the claim holds for some $1 \le n < N$. If $g_{k_1, l}(n) \not\in \Zp{p}$, then both of $g_{k_1, l}(n+1)$ and $g_{k_2, l}(n+1)$ are not in $\Zp{p}$ by the induction hypothesis and (2). On the other hand, if $g_{k_1, l}(n) \in \Zp{p}$, then
	\begin{align*}
		(n+1) \bigg(g_{k_1, l}(n+1) - g_{k_2, l}(n+1) \bigg) &= n (g_{k_1, l}(n) - g_{k_2, l}(n)) + (g_{k_1, l}(n)^{k_1} - g_{k_2, l}(n)^{k_2}).
	\end{align*}
	By the induction hypothesis, the first term is in $p^{r-\nu_p(n!)} \Zp{p}$. As for the second term, if $g_{k_1,l}(n) \not\in p\Zp{p}$, then by applying Euler's theorem, it belongs to $p^{r-\nu_p(n!)} \Zp{p}$. If $g_{k_1, l}(n) \in p\Zp{p}$, then since $k_1, k_2 \ge r$, it is also in $p^{r-\nu_p(n!)} \Zp{p}$. Therefore, by dividing the both sides by $(n+1)$, we get $g_{k_1, l}(n+1) - g_{k_2, l}(n+1) \in p^{r - \nu_p((n+1)!)} \Zp{p}$. In particular, if $g_{k_1, l}(n+1) \not\in \Zp{p}$, then $g_{k_2, l}(n+1) \not\in \Zp{p}$, and vice versa.
	
	(4) Since $g_{k, l_1}(n)$ and $g_{k, l_2}(n)$ satisfy the same recursion with the same initial value modulo $p^r$, the claim immediately follows in a similar manner to (3).
\end{proof}

\subsection{Proof of \cref{Main1}}

What we have to check are the following two claims.
\begin{enumerate}
	\item For any $k, l \ge 2$ and $1 \le n \le 6$, $g_{k,l}(n) \in \Z$.
	\item $g_{k,l}(7) \not\in \Z$ if and only if $k \equiv 2 \pmod{6}$ and $l \equiv 3 \pmod{7}$.
\end{enumerate}

By applying \cref{lem:key-properties} with $N=7$ and combining it with \eqref{Z-Zp}, these claims can be translated as follows.
\begin{lem}\label{Main1-another}
	\cref{Main1} is equivalent to the following claims for $p = 2,3,5,7$.
	\begin{enumerate}
		\item For $2 \le k \le 11$ and $1 \le l \le 16$, we have $g_{k,l}(7) \in \Zp{2}$.
		\item For $2 \le k \le 7$ and $1 \le l \le 9$, we have $g_{k,l}(7) \in \Zp{3}$.
		\item For $1 \le k \le 4$ and $1 \le l \le 5$, we have $g_{k,l}(7) \in \Zp{5}$.
		\item For $1 \le k \le 6$ and $1 \le l \le 7$, $g_{k,l}(7) \not\in \Zp{7}$ if and only if $k=2$ and $l=3$.
	\end{enumerate}
\end{lem}

\begin{proof}
	Since it is obvious that \cref{Main1} implies the claims, we will now show the converse implication. By \cref{lem:key-properties} (1), $g_{k,l}(7) \in \Zp{p}$ for any $p > 7$. Since the discussion remains similar for the remaining cases of $p=2,3,5$, and 7, let us focus here on explaining the case when $p=2$. By \cref{lem:key-properties} (2), $g_{k,l}(7) \in \Zp{2}$ if and only if $g_{k,l}(n) \in \Zp{2}$ for all $1 \le n \le 7$. By the periodicity shown in \cref{lem:key-properties} (3) and (4), it is enough to show that $g_{k,l}(7) \in \Zp{2}$ for $k = 2,3$, and $4 \le k < 4 + \varphi(2^4)$, and $1 \le l \le 2^4$, where we note that $\nu_2(7!) = 4$. This is the claim for $p=2$.
\end{proof}

\begin{proof}[Proof of \cref{Main1}]
	Since there are only a finite (and relatively small) number of cases to consider, it can be checked by using Mathematica. The codes to compute them are available in \cref{s-app}. Alternatively, it is enough (and possible) to check that $g_{k,l}(7) \in \Z$ if and only if $(k,l) \neq (2,3), (2,10), (8,3), (8,10)$ for $1 \le k \le 11$ and $1 \le l \le 16$ because $g_{k,l}(7)$ is not too large in these cases.
\end{proof}

\section{Zagier's asymptotic formula and its generalization} \label{s3}

In this section, we prove \cref{thm:generalized-zagier-asymptotic} by defining the constant $C_{k,l}$ and the asymptotic coefficients $a_{k,r}$ explicitly.

\subsection{The constant $C_{k,l}$}

We first show a monotonic behavior of $C_{k,l}(n) \coloneqq g_{k,l}(n)^{1/k^n}$.

\begin{figure}[H]
\centering
	\begin{subfigure}{0.3\columnwidth}
		\includegraphics[width=\columnwidth]{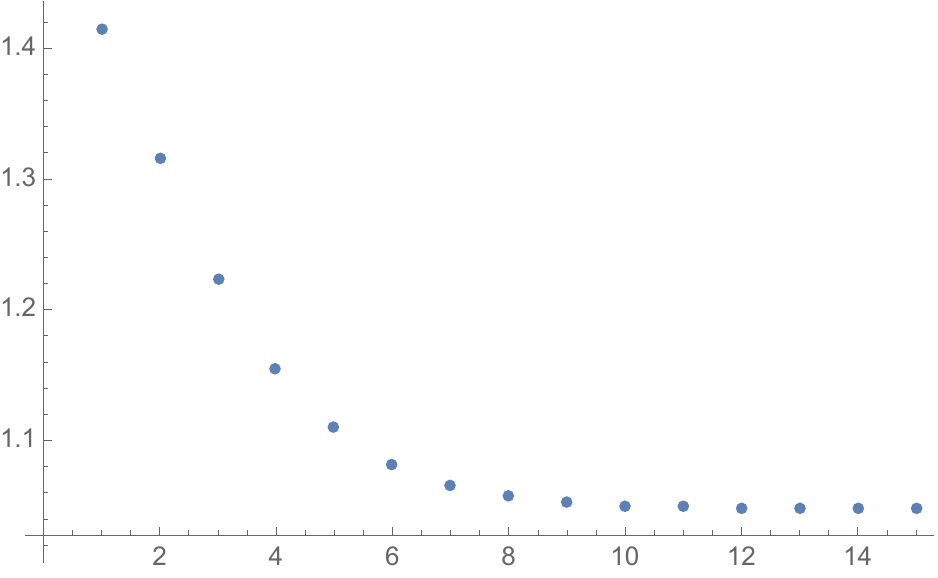}
		\caption*{$(k,l)=(2,2)$}
	\end{subfigure}
	\begin{subfigure}{0.3\columnwidth}
		\includegraphics[width=\columnwidth]{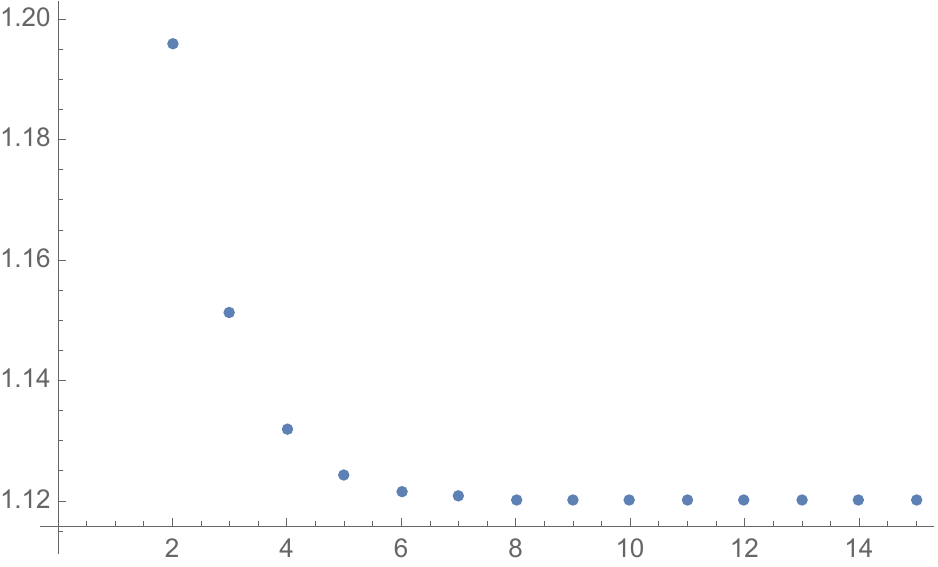}
		\caption*{$(k,l)=(3,2)$}
	\end{subfigure}
	\begin{subfigure}{0.3\columnwidth}
		\includegraphics[width=\columnwidth]{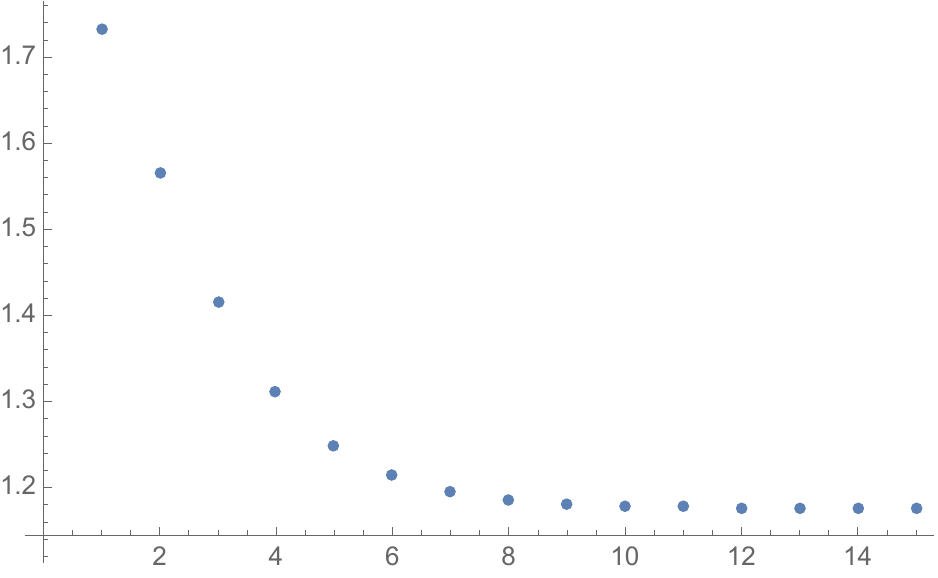}
		\caption*{$(k,l)=(2,3)$}
	\end{subfigure}
	\caption{The plots of $C_{k,l}(n)$ for $1\leq n\leq 15$.}
	\label{fig:plot-d}
\end{figure}

\begin{lem}
	For any integers $k,l \ge 2$ and $n \ge 1$, we have $C_{k,l}(n) > C_{k,l}(n+1) > 1$.
\end{lem}

\begin{proof}
	First, we check that $C_{k,l}(n) > 1$, that is, $g_{k,l}(n) > 1$ by induction on $n$. The initial condition is satisfied by $g_{k,l}(1) = l \ge 2$. Assume that $g_{k,l}(n) > 1$. Then, we have
	\[
		g_{k,l}(n+1) = \frac{1}{n+1} \bigg(n g_{k,l}(n) + g_{k,l}(n)^k \bigg) > \frac{n+1}{n+1} = 1.
	\]
	Next, we check for monotonicity. Since $g_{k,l}(n) > 1$, we have $g_{k,l}(n)^{k-1} > 1$, and then
	\[
		\frac{1}{n+1} \left(\frac{n}{g_{k,l}(n)^{k-1}} + 1 \right) < 1.
	\]
	Hence, we have
	\begin{align}\label{Ckl-rec}
		\left(\frac{C_{k,l}(n+1)}{C_{k,l}(n)}\right)^{k^{n+1}} = \frac{g_{k,l}(n+1)}{g_{k,l}(n)^k} = \frac{1}{n+1} \frac{n g_{k,l}(n) + g_{k,l}(n)^k}{g_{k,l}(n)^k} < 1,
	\end{align}
	which concludes the proof.
\end{proof}

The above lemma shows that $C_{k,l}(n)$ converges. We denote the limit as
\begin{align}\label{dfn:C-kl}
	C_{k,l} \coloneqq \lim_{n \to \infty} C_{k,l}(n) \ge 1.
\end{align}

Next, we introduce some notations to show the strict inequality $C_{k,l} > 1$. For $k \ge 2$, we define the \emph{$k$-Somos constant} $\sigma_k$ by
\begin{align}\label{k-Somos}
	\sigma_k \coloneqq \exp \left(\sum_{m=1}^\infty \frac{\log m}{k^m} \right) > 1,
\end{align}
(see also~\cite{Sondow2007}). For a real number $t_0 > 1$, we define a Somos-like sequence $t_k(n)$ by the recursion:
\begin{align}\label{tk-rec}
	t_k(n+1) = \frac{1}{n+1} t_k(n)^k
\end{align}
with the initial value $t_k(0) = t_0$.

\begin{lem}\label{lem:tkn>1}
	If $t_0 > \sigma_k$, then there exist a constant $c > 1$ such that $t_k(n)^{1/k^n}$ decreases monotonically and tends to $c$ as $n \to \infty$.
\end{lem}

\begin{proof}
	It is equivalent to show that $a_k(n) \coloneqq k^{-n} \log t_k(n)$ decreases and tends to a positive constant. Since $a_k(n)$ satisfies the recursion
	\[
		a_k(n) = a_k(n-1) - \frac{\log n}{k^n}
	\]
	for $n \ge 1$, the sequence decreases monotonically. By our assumption, we have $a_k(0) > \log \sigma_k$, which implies that
	\[
		\lim_{n \to \infty} a_k(n) = a_k(0) - \sum_{n=1}^\infty \frac{\log n}{k^n} > 0.
	\]
\end{proof}

We can estimate $(k,l)$-G\"{o}bel sequence $g_{k,l}(n)$ by using the sequence $t_k(n)$ as follows.

\begin{lem}\label{lem:g>t}
	If $l \ge t_0^k$, then $g_{k,l}(n) \ge t_k(n)$ for all $n \ge 1$.
\end{lem}

\begin{proof}
	It immediately follows from induction on $n$. Indeed, we have $g_{k,l}(1) = l \ge t_0^k = t_k(1)$. Assume that $g_{k,l}(n) \ge t_k(n)$. Then we obtain
	\[
		g_{k,l}(n+1) > \frac{1}{n+1} g_{k,l}(n)^k \ge \frac{1}{n+1} t_k(n)^k = t_k(n+1).
	\]
\end{proof}

\begin{prop}
	For any integers $k, l \ge 2$, the constant $C_{k,l}$ defined in \eqref{dfn:C-kl} satisfies $C_{k,l} > 1$.
\end{prop}

\begin{proof}
	We show it by considering three cases. 
	
	(1) For $k \ge 3$ and $l \ge 2$, since $\sigma_3^3 = 1.5462\ldots$ and $\sigma_k^k$ decreases monotonically, there exists a real number $t_0 > 1$ such that $l \ge t_0^k > \sigma_k^k$. For such a $t_0$, by applying \cref{lem:tkn>1} and \cref{lem:g>t}, we obtain
	\[
		C_{k,l}(n) = g_{k,l}(n)^{1/k^n} \ge t_k(n)^{1/k^n} \ge c > 1,
	\]
	which implies that $C_{k,l} > 1$.
	
	(2) For $k=2$ and $l \ge 3$, since $\sigma_2^2 = 2.7612\ldots$, there exists a real number $t_0 > 1$ such that $l \ge t_0^2 > \sigma_2^2$. By the same argument, we also obtain $C_{k,l} > 1$.
	
	(3) For $k = l = 2$, we need to modify the argument. We define another sequence $t'_2(n)$ by the same recursion as in \eqref{tk-rec} with the initial value $t'_2(3) = 5$. Then, the inequality $g_{2,2}(n) \ge t'_2(n)$ for $n \ge 3$ is shown in a manner similar to \cref{lem:g>t}. Thus, we obtain
	\[
		C_{2,2}(n) = g_{2,2}(n)^{1/2^n} \ge t'_2(n)^{1/2^n}.
	\]
	and
	\[
		\lim_{n \to \infty} \frac{1}{2^n} \log t'_2(n) = \frac{1}{2^3} \log 5 - \sum_{n=4}^\infty \frac{\log n}{2^n} = 0.00395\ldots > 0.
	\]
	Hence, we conclude that $C_{2,2} > 1$.
\end{proof}

\subsection{Asymptotic behavior}

In the previous subsection, we defined the constant $C_{k,l} > 1$. By definition, we obtain
\begin{align}\label{eq:relatin-g-c-d}
	\frac{g_{k,l}(n)}{C_{k,l}^{k^n}} = \left(\frac{C_{k,l}(n)}{C_{k,l}}\right)^{k^n}.
\end{align}
Thus, it is sufficient to evaluate the right-hand side to prove \cref{thm:generalized-zagier-asymptotic}. The aim of this subsection is to establish a connection to a simpler sequence.

\begin{thm}\label{thm:error-conv}
	For any real number $R > 0$, we have
	\begin{align*}
		\exp \left(\sum_{m=1}^\infty\frac{\log{(m+n)}}{k^m}\right) - \left(\frac{C_{k,l}(n)}{C_{k,l}}\right)^{k^n} =O\left(\frac{1}{n^R}\right) 
	\end{align*}
	as $n \to \infty$.
\end{thm}

First, we prepare a lemma with the aim of proving this theorem. Let
\[
	 \epsilon_{k,l}(n) \coloneqq \sum_{m=1}^\infty\frac{\log{(m+n)}}{k^m} - k^n (\log C_{k,l}(n) - \log C_{k,l}).
\]

\begin{lem}\label{epsilon-bound}
	We have
	\[
		\epsilon_{k,l}(n) = \sum_{m=1}^\infty \frac{1}{k^m} \log \left(1+ \frac{m+n-1}{g_{k,l}(m+n-1)^{k-1}} \right) < \frac{2n}{\exp(k^{n-1})}
	\]
	for sufficiently large $n$.
\end{lem}

\begin{proof}
	By \eqref{Ckl-rec}, we have
	\[
		\log C_{k,l}(m+n) - \log C_{k,l}(m+n-1) = \frac{1}{k^{m+n}} \log \left(\frac{m+n-1}{g_{k,l}(m+n-1)^{k-1}} + 1 \right) - \frac{\log (m+n)}{k^{m+n}}.
	\]
	By summing each side over $m$, we obtain
	\[
		k^n (\log C_{k,l} - \log C_{k,l}(n)) = \sum_{m=1}^\infty \frac{1}{k^m} \log \left(\frac{m+n-1}{g_{k,l}(m+n-1)^{k-1}} + 1 \right) - \sum_{m=1}^\infty \frac{\log (m+n)}{k^m},
	\]
	which implies the first equality. 
	
	Next, the inequality $\log(1+x) < x$ for $x >0$ implies that
	\[
		\sum_{m=1}^\infty \frac{1}{k^m} \log \left(1+ \frac{m+n-1}{g_{k,l}(m+n-1)^{k-1}} \right) < \sum_{m=1}^\infty \frac{1}{k^m} \frac{m+n-1}{g_{k,l}(m+n-1)^{k-1}}.
	\]
	Since $C_{k,l}(n) \ge C_{k,l} > 1$ and $g_{k,l}(n) = C_{k,l}(n)^{k^n}$, we have
	\[
		g_{k,l}(m+n-1)^{k-1} = C_{k,l}(m+n-1)^{(k-1) k^{m+n-1}} \ge C_{k,l}^{(k-1)k^n} > \exp(k^{n-1})
	\]
	for sufficiently large $n$. Therefore, by using $m + n - 1 \le mn$ for $m, n \ge 1$, we obtain
	\[
		\epsilon_{k,l}(n) < \frac{n}{\exp(k^{n-1})} \sum_{m=1}^\infty \frac{m}{k^m} = \frac{n}{\exp(k^{n-1})} \frac{k}{(k-1)^2} \le \frac{2n}{\exp(k^{n-1})}.
	\]
\end{proof}

\begin{proof}[Proof of \cref{thm:error-conv}]
	By using the expression
	\[
		\left(\frac{C_{k,l}(n)}{C_{k,l}}\right)^{k^n} = \exp \left(\sum_{m=1}^\infty \frac{\log(m+n)}{k^m} - \epsilon_{k,l}(n) \right)
	\]
	and the inequality $e^x - e^{x-\epsilon} \le \epsilon e^x$, we have
	\[
		\exp{\left(\sum_{m=1}^\infty\frac{\log{(m+n)}}{k^m}\right)} - \left(\frac{C_{k,l}(n)}{C_{k,l}}\right)^{k^n} \le \epsilon_{k,l}(n) \exp\left(\sum_{m=1}^\infty \frac{\log(m+n)}{k^m} \right).
	\]
	Moreover, by applying $m+n \le m(n+1)$ for $m,n \ge 1$, it is bounded by
	\[
		\le \epsilon_{k,l}(n) \sigma_k \cdot \exp\left(\sum_{m=1}^\infty \frac{\log(n+1)}{k^m}\right) = \epsilon_{k,l}(n) \sigma_k \cdot (n+1)^{\frac{1}{k-1}},
	\]
	where $\sigma_k$ is the $k$-Somos constant defined in \eqref{k-Somos}. Finally, by \cref{epsilon-bound}, we obtain the theorem.
\end{proof}

\subsection{The asymptotic coefficients $a_{k,r}$}

Finally, to complete the statement of \cref{thm:generalized-zagier-asymptotic} and its proof, we provide the asymptotic expansion of the first term of \cref{thm:error-conv}. First, we review the relevant parts of the studies by Sondow and Hadjicostas~\cite{Sondow2007} regarding a generalization of Somos's sequence introduced in \eqref{somos-seq}. Then, we establish a connection between their results and our $(k,l)$-G\"{o}bel sequences. To explicitly state their claims, we recall the Eulerian polynomials.

\begin{dfn}
	For any integer $r \ge 0$, we define the \emph{Eulerian polynomial} $A_r(t)\in\mathbb Z{[t]}$ by
	\begin{align*}
		\sum_{m=1}^\infty\frac{m^r}{t^m}=\frac{A_r(t)}{(t-1)^{r+1}}.
	\end{align*}
\end{dfn}
  
\begin{exm}
	The first few examples are given by $A_0(t) =1$ and
	\begin{align*}
		A_1(t)&=t,\\
		A_2(t)&=t^2+t,\\
		A_3(t)&=t^3+4t^2+t,\\
		A_4(t)&=t^4+11t^3+11t^2+t,\\
		A_5(t)&=t^5+26t^4+66t^3+26t^2+t.
	\end{align*}
\end{exm}

Then, the following is known.

\begin{thm}[\protect{\cite[Theorem 9 and Lemma 1]{Sondow2007}}]\label{thm:a-dfn-property}
	We define the sequence $(a_{k,r})_r$ to be
	\begin{align}\label{def-akr}
		a_{k,r}\coloneqq\sum_{\substack{m_1, \dots, m_r \geq 0 \\ m_1 + 2m_2 + \cdots + r m_r = r}}\prod_{j=1}^r\frac{1}{m_j !}\left(\frac{(-1)^{j-1}}{j}\frac{A_j(k)}{(k-1)^{j+1}}\right)^{m_j}, 
	\end{align}
	then we have
	\begin{align*}
		\exp\left(\sum_{m=1}^\infty\frac{\log(m+n)}{k^m}\right)\sim n^{\frac{1}{k-1}}\left(1+\sum_{r=1}^\infty\frac{a_{k,r}}{n^r}\right) \quad (n\to\infty). 
	\end{align*}
\end{thm}

\begin{exm}
	The first several terms are calculated as follows.
	\begin{align*}
		a_{k,1} &=\frac{k}{(k-1)^2},\\
		a_{k,2} &=-\frac{k(k^2 - k - 1)}{2(k-1)^4},\\
		a_{k,3} &=\frac{k(2k^4 + k^3 - 11k^2 + 7k + 2)}{6(k-1)^6},\\
		a_{k,4} &=-\frac{k(6k^6 + 37k^5 - 124k^4 + 53k^3 + 92k^2 - 59k - 6)}{24(k-1)^8},\\
		a_{k,5} & = \frac{k(24k^8 + 478k^7 - 1013k^6 - 1324k^5 + 4411k^4 - 2724k^3 - 453k^2 + 578k + 24)}{120(k-1)^{10}}.
\end{align*}
In particular, when $k=2$, we observe that $(a_{2,r})_{r=1}^5=(2,-1,4,-21,138)$, which matches the asymptotic coefficients in \eqref{Zagier-asymp}. 
\end{exm}

\begin{proof}[Proof of \cref{thm:generalized-zagier-asymptotic}]
	We will show that 
	\begin{align*}
		\frac{g_{k,l}(n)}{C_{k,l}^{k^n}}-n^{\frac{1}{k-1}}\left(1+\sum_{r=1}^R \frac{a_{k,r}}{n^r}\right)=O\left(\frac{n^{\frac{1}{k-1}}}{n^{R+1}}\right) \quad (n\to\infty)
	\end{align*}
	for any $R \ge 0$. It immediately follows by applying \eqref{eq:relatin-g-c-d}, \cref{thm:error-conv}, and \cref{thm:a-dfn-property}.
\end{proof}

\section{Further observations} \label{s4}

Zagier~\cite{Zagier1996} observed not only the asymptotic formula but also a heuristic explaining why $N_{2,2} = 43$ is unexpectedly large, assuming a certain ``randomness" of the values $g_{2,2}(n)$ modulo $p$ for $1 \le n < p$. Inspired by his heuristic argument, we can ask the question: for any pair of integers $k, l \ge 2$, does there exist (infinitely many) $p \in \mathcal{P}$ such that $g_{k,l}(p) \not\in \Zp{p}$. For instance, $g_{2,2}(p) \not\in \Zp{p}$ holds when $p = 43, 61, 67, 83, \ldots$. We do not have an answer to this question, but we obtained a result concerning ``randomness", which we present as a final remark.

\begin{thm}\label{random}
	For any prime number $p \in \mathcal{P}$ and integers $k, l, r \ge 2$, the set
	\[
		G_{k,l,p}^r \coloneqq \{g_{k,l}(n) \bmod{p^r} : 1 \le n < p, g_{k,l}(n) \equiv 0 \pmod{p^{r-1}}\}
	\]
	is a singleton $\{0 \bmod{p^r}\}$ or has the same cardinality as $\{1 \le n < p : g_{k,l}(n) \equiv 0 \pmod{p^{r-1}}\}$. Here the congruence is considered in $\Zp{p}$.
\end{thm}

\begin{exm}
	Let $r = 2$. For $(k,l) = (4,4)$ and $p=13$, we see that $g_{4,4}(1) \not\equiv 0 \pmod{13}$ and $g_{4,4}(n) \equiv 0 \pmod{13}$ for $2 \le n \le 12$. Moreover, we can check that all entries of
	\[
		(g_{4,4}(n) \bmod{13^2})_{2 \le n \le 12} = (130,143, 65, 52, 156, 13, 117, 104, 26, 39, 78)
	\]
	are distinct from each other.
	
	On the other hand, for $(k,l) = (3,2)$ and $p = 13$, we see that $g_{3,2}(n) \not\equiv 0 \pmod{13}$ for $1 \le n \le 3$ and $g_{3,2}(n) \equiv 0 \pmod{13^2}$ for $4 \le n \le 12$. Thus, $G_{3,2,13}^2 = \{0 \bmod{13^2}\}$ is a singleton. 
\end{exm}

To prove the theorem, we first show a lemma.

\begin{lem}\label{random-lem}
	Let $k,l,r \ge 2$ be integers and $p \in \mathcal{P}$. Assume that there exists $1 \le a < p$ and $0 \le b < p$ such that $g_{k,l}(a) \equiv b p^{r-1} \pmod{p^r}$. Then, for any $a \le n < p$, we have
	\[
		n g_{k,l}(n) \equiv ab p^{r-1} \pmod{p^r}.
	\]
\end{lem}

\begin{proof}
	It follows from induction on $n$. The first condition $a g_{k,l}(a) \equiv ab p^{r-1} \pmod{p^r}$ is clearly satisfied by our assumption. Assume that $ng_{k,l}(n) \equiv ab p^{r-1} \pmod{p^r}$ for some $a \le n < p-1$. By definition,
	\[
		(n+1)g_{k,l}(n+1) = g_{k,l}(n)^k + n g_{k,l}(n) \equiv (n^{-1} ab)^k p^{k(r-1)} + ab p^{r-1} \pmod{p^r},
	\]
	which implies that $(n+1) g_{k,l}(n+1) \equiv ab p^{r-1} \pmod{p^r}$ because $k(r-1) \ge r$.
\end{proof}

\begin{proof}[Proof of \cref{random}]
	It is sufficient to consider the case where $I \coloneqq \{1 \le n < p : g_{k,l}(n) \equiv 0 \pmod{p^{r-1}}\} \neq \emptyset$. Then, we take $a \coloneqq \min I$ and $0 \le b < p$ satisfying $g_{k,l}(a) \equiv b p^{r-1} \pmod{p^r}$. If $b = 0$, then \cref{random-lem} tells us that $g_{k,l}(n) \equiv 0 \pmod{p^r}$ for all $n \in I$, that is, $G_{k,l,p}^r = \{0 \bmod{p^r}\}$. We now assume that $b \neq 0$. If there exist $n_1, n_2 \in I$ satisfying $g_{k,l}(n_1) \equiv g_{k,l}(n_2) \pmod{p^r}$, \cref{random-lem} implies that
	\[
		n_1 g_{k,l}(n_1) \equiv ab p^{r-1} \equiv n_2 g_{k,l}(n_2) \pmod{p^r},
	\]
	and then $n_1 = n_2$.
\end{proof}

Furthermore, \cref{tab:Ibstedt} suggests a tendency for $N_{k,l}$ to increase when $k$ is prime, but elucidating this phenomenon will be a subject for future investigation.

\appendix
\section{Methods for Computing $N_{k,l}$} \label{s-app}

In this appendix, we provide a method to compute $N_{k,l}$ and the Mathematica codes we used to generate \cref{tab:Ibstedt}. First, we recall the following sequence $g_{k,l,p,r}(n)$ introduced in~\cite{MMT2024}.

\begin{dfn}
Let $k,l \geq 2, r \geq 1$ be integers, and $p$ a prime. We define $b(n) = b_{p,r}(n) = r - \nu_p(n!)$ and use the symbol $\mathsf{F}$ to represent ``False." For any positive integer $n$ with $\nu_p(n!) \leq r$, we define $g_{k,l,p,r}(n) \in \mathbb{Z}/p^{b(n)}\mathbb{Z} \cup \{\mathsf{F}\}$ as follows.
\begin{itemize}
	\item[$\bullet$] Initial condition: $g_{k,l,p,r}(1) = l \bmod{p^r} \in \mathbb{Z}/p^r \mathbb{Z}$.
	\item[$\bullet$] For $n \geq 2$: When $g_{k,l,p,r}(n-1) = \mathsf{F}$, $g_{k,l,p,r}(n) = \mathsf{F}$.
	\item[$\bullet$] For $n \geq 2$: When $g_{k,l,p,r}(n-1) = a \bmod{p^{b(n-1)}}$, 
	\begin{itemize}
		\item[-] if $a(n-1+a^{k-1}) \not\equiv 0 \pmod{p^{\nu_p(n)}}$, then $g_{k,l,p,r}(n) = \mathsf{F}$.
		\item[-] if $a(n-1+a^{k-1}) \equiv 0 \pmod{p^{\nu_p(n)}}$, then letting 
		$c \in \mathbb{Z}$ such that $c \cdot (n/p^{\nu_p(n)}) \equiv 1 \pmod{p^{b(n)}}$, we define
	\[
		g_{k,l,p,r}(n) = \frac{a(n-1+a^{k-1})}{p^{\nu_p(n)}} \cdot c \bmod{p^{b(n)}}.
	\]
	\end{itemize}
	\end{itemize}
\end{dfn}

The recursion above, defining $g_{k,l,p,r}(n)$, is a translation of \cref{dfn:kl-Goebel} modulo a power of $p$. Since $g_{k,l}(n) \not\in \Zp{p}$ if and only if $g_{k,l,p,\nu_p(n!)}(n) = \mathsf{F}$, $N_{k,l}$ can be expressed as
\[
    N_{k,l} = \inf \{n \in \Z_{>0} : \text{there exists $p \in \mathcal{P}_{\le n}$ such that } g_{k,l,p,\nu_p(n!)}(n) = \mathsf{F}\}.
\]
The following codes implement this argument in Mathematica.\\

\noindent\hrulefill
\begin{verbatim}
nu[p_, n_] := FirstCase[FactorInteger[n], {p, r_} -> r, 0];
inv[n_, M_] := If[M == 1, 1, ModularInverse[n, M]];

g[k_, l_, p_, r_, 1] := {Mod[l, p^r], r};
g[k_, l_, p_, r_, n_] := 
  g[k, l, p, r, n] = 
   Module[{a, b}, 
    If[g[k, l, p, r, n - 1] === "F", 
     "F", {a, b} = g[k, l, p, r, n - 1];
     If[Mod[a (n - 1 + a^(k - 1)), p^nu[p, n]] != 0, 
      "F", {Mod[
        a (n - 1 + a^(k - 1))/p^nu[p, n] inv[n/p^nu[p, n], 
          p^(b - nu[p, n])], p^(b - nu[p, n])], b - nu[p, n]}]]];       
          
NN[k_, l_] := 
 Module[{n}, n = 2; 
  While[Not[
    MemberQ[Table[
      g[k, l, Prime[m], nu[Prime[m], n!], n], {m, 1, PrimePi[n]}], 
     "F"]], n++]; n];
\end{verbatim}
\noindent\hrulefill

Here, the output \verb|g[k,l,n,p,r] = {a, b}| means that $g_{k,l,p,r}(n) = a \bmod{p^b}$ with $b = b(n)$, that is, $g_{k,l}(n) \equiv a \pmod{p^b}$, and \verb|NN[k,l]| gives the value of $N_{k,l}$.

\bibliographystyle{amsplain}
\bibliography{References} 

\end{document}